\documentclass[12pt]{amsart}
\usepackage[a4paper,margin=1in]{geometry}
\usepackage{amsmath,amssymb,amsthm}
\usepackage{hyperref}
\usepackage{enumitem}

\title{Edge partitions into induced-$2K_2$-free bipartite graphs}
\author{András London}
\date{\today}

\address{Institute of Informatics, University of Szeged, Szeged, Hungary}
\email{london@inf.u-szeged.hu \\ ORCID: 0000-0003-1957-5368}

\newtheorem{theorem}{Theorem}
\newtheorem{proposition}{Proposition}
\newtheorem{lemma}[theorem]{Lemma}
\newtheorem{corollary}[theorem]{Corollary}
\newtheorem{definition}[theorem]{Definition}
\newtheorem{remark}[theorem]{Remark}
\newtheorem{problem}[theorem]{Problem}

\newcommand{\fp}{\operatorname{fp}}
\newcommand{\ind}{\operatorname{ind}}
\newcommand{\width}{\operatorname{width}}
\newcommand{\NP}{\operatorname{NP}}

\begin{document}
\maketitle

\begin{abstract}
We study edge partitions of a bipartite graph into induced-$2K_2$-free bipartite graphs, i.e.\ into Ferrers (chain) graphs.
We define $\fp(G)$ as the minimum number of parts in such a partition.
We prove general lower and upper bounds in terms of induced matchings and Dilworth widths of neighborhood posets.
We compute the parameter exactly for paths and even cycles, and we exhibit separations showing that the induced-matching lower bound and the width upper bound can both be far from tight.
We also record a simple host-induced conflict-graph lower bound, present a $0$--$1$ matrix viewpoint, and add some complexity remarks.
\end{abstract}

\section{Introduction}

Edge partitions into graphs from a hereditary class provide an edge-analogue of graph colorings.
In a recent template-partition framework, Gy\H{o}rffy, London, Nagy and Pluh\'ar \cite{GyorffyLondonNagyPluhar2025} define, for a fixed graph $H$,
a parameter $\chi'_H(G)$ as the minimum number of bipartite subgraphs whose edge sets partition $E(G)$, each avoiding $H$ as an induced subgraph.
In this paper we focus on the first nontrivial bipartite case when $H=2K_2$.
The induced-$2K_2$-free bipartite graphs are exactly the Ferrers graphs, also called chain graphs:
equivalently, the neighborhoods on one side are totally ordered by inclusion.
They admit a convenient $0$--$1$ matrix description: after permuting rows and columns of the corresponding bipartite adjacency matrix, the $1$-entries form a Ferrers (Young-diagram) shape.
Background and further characterizations can be found in \cite{BrandstadtLeSpinrad}; see also \cite{EhrenborgVW}.

Given a bipartite graph $G=(U,V,E)$, we write $\fp(G)$ for the minimum $k$ such that
$E$ can be partitioned into $k$ edge-disjoint Ferrers subgraphs on the same bipartition.
For bipartite host graphs this is exactly the specialization
\[
\fp(G)=\chi'_{2K_2}(G)
\]
from \cite{GyorffyLondonNagyPluhar2025}.
Our parameter is distinct from Ferrers dimension and from chain-subgraph covers, where overlaps are allowed.

\begin{center}
\renewcommand{\arraystretch}{1.15}
\begin{tabular}{|l|c|c|l|}
\hline
\textbf{model} & \textbf{operation} & \textbf{overlaps} & \textbf{typical constraint} \\
\hline
Ferrers dimension~\cite{Cogis,ChatterjeeGhosh} & intersection & yes & intersection yields $G$ \\
chain-subgraph cover~\cite{MaSpinrad,BrandstadtES} & union cover & yes & each piece is Ferrers \\
Ferrers partition number $\fp$~\cite{GyorffyLondonNagyPluhar2025} & edge partition & no & each piece is Ferrers \\
\hline
\end{tabular}
\end{center}

\medskip
\noindent\textbf{Main results.}
\begin{itemize}[leftmargin=2em]
\item We prove general bounds
\[
\nu_{\ind}(G)\le \fp(G)\le \min\{\width(\mathcal P_U),\,\width(\mathcal P_V)\},
\]
in terms of induced matching number $\nu_{\ind}(G)$ and Dilworth widths of neighborhood-inclusion posets (Lemma~\ref{lem:lower} and Theorem~\ref{thm:upper}).

\item We compute $\fp$ exactly for paths and even cycles (Theorems~\ref{thm:path} and \ref{thm:cycle}).

\item We provide separation examples: disjoint unions of $C_8$ give an arbitrarily large gap $\fp(G)-\nu_{\ind}(G)$ (Corollary~\ref{cor:gap}),
while crown graphs show that the width upper bound can be off by a factor $\Theta(n)$ (Theorem~\ref{thm:crown}).

\item In the matrix viewpoint, we determine $\fp$ exactly for $K_{m,n}$ with a matching removed:
$\fp(K_{m,n}\setminus M)\in\{1,2\}$ with a sharp threshold at $|M|=2$ (Theorem~\ref{thm:Kmn-minus-matching}).

\item We also record a simple host-induced $2K_2$ conflict-graph lower bound and give a worked example illustrating the definitions (Section~\ref{sec:conflict}).
\end{itemize}

\medskip
\noindent The paper is organized as follows.
Section~\ref{sec:related} discusses connections to earlier work.
Sections~\ref{sec:bounds}--\ref{sec:separations} develop bounds and examples, and Section~\ref{sec:matrix} presents the matrix viewpoint.
We close with complexity remarks and open problems.

\section{Related Work}
\label{sec:related}

Ferrers graphs appear throughout the literature under the names \emph{chain graphs} and \emph{difference graphs}; see the survey monograph \cite{BrandstadtLeSpinrad}.
They admit the familiar forbidden induced subgraph characterization, namely the $2K_2$, and a parallel characterization by Ferrers-shaped $0$--$1$ matrices.

Two nearby lines of work should be distinguished from the present paper.
First, \emph{Ferrers dimension} and related intersection parameters ask for the minimum number of Ferrers graphs whose \emph{intersection} yields a given (di)graph; see Cogis \cite{Cogis} and, for connections to boxicity, \cite{ChatterjeeGhosh}.
Second, \emph{chain subgraph covers} ask for a minimum collection of chain subgraphs whose \emph{union} covers all edges, allowing overlaps; see \cite{MaSpinrad} and later developments for restricted bipartite classes \cite{BrandstadtES}.
Our parameter is a partition analogue: overlaps are forbidden, so the combinatorics changes.

Edge partitions into constrained bipartite subgraphs have also been studied recently in a broader ``template class'' framework; see \cite{GyorffyLondonNagyPluhar2025} for a representative instance motivated by forbidden-subgraph constraints. A related motivation for restricting the structure between pieces also appears in the ``generalized coloring'' model of London, Martin and Pluh\'ar  \cite{LondonMartinPluhar2022}.

Finally, we should mention Gera et al. works \cite{GeraLondonPluhar2022, GeraLondon2023} on exploring a more application-motivated and algorithmic approach in connection with the concept of \emph{nestedness}, a phenomenon that often emerges in ecological and real transaction networks.

\section{Definitions}

All graphs are finite and simple. Let $G=(U,V,E)$ be bipartite with color classes $U$ and $V$. All subgraphs considered below are spanning on the same bipartition $(U,V)$; only the edge set is partitioned (isolated vertices are allowed in parts).
We write $2K_2$ for the graph on four vertices consisting of two disjoint edges.

\begin{definition}
A bipartite graph is a \emph{Ferrers graph} (chain graph) if it contains no induced copy of $2K_2$.
Equivalently, the family (of neighborhoods) $\{N(u):u\in U\}$ (or symmetrically $\{N(v):v\in V\}$) is totally ordered by inclusion.
\end{definition}

\begin{proposition}\label{prop:one-component}
A Ferrers graph has at most one connected component containing an edge.
Equivalently, if a bipartite graph contains two vertex-disjoint edges in distinct connected components,
then it contains an induced $2K_2$.
\end{proposition}

\begin{proof}
Let $e=uv$ and $f=u'v'$ be edges in different connected components.
There are no edges between $\{u,v\}$ and $\{u',v'\}$, hence the subgraph induced by
$\{u,v,u',v'\}$ consists of exactly the two edges $e$ and $f$, i.e.\ an induced $2K_2$.
Therefore a Ferrers graph (which is induced-$2K_2$-free) cannot contain edges in two distinct components.
\end{proof}

\begin{definition}
For a bipartite graph $G$, the \emph{Ferrers partition number} $\fp(G)$ is the minimum $k$ such that
$E(G)$ can be partitioned into a disjoint union of $k$ edge sets, each inducing a Ferrers graph on the same bipartition, i.e.
\[
E = E_1 \dot\cup \cdots \dot\cup E_k,
\]
where each $G_i=(U,V,E_i)$ is a Ferrers graph.
\end{definition}

This parameter reacts sharply to induced $2K_2$ obstructions, and its behavior differs from cover/intersection models.
Even for very classical families (paths, cycles, crown graphs) one sees both tightness and large gaps.

\begin{remark}
Gy\H{o}rffy et al.~\cite{GyorffyLondonNagyPluhar2025} define $\chi'_H(G)$ as the minimum size of an edge partition of $G$ into bipartite graphs avoiding $H$ as an induced subgraph.
For bipartite $G$ and $H=2K_2$ we have $\fp(G)=\chi'_{2K_2}(G)$.
\end{remark}

\begin{definition}
Given graph $G=(V,E)$, a set $M\subseteq E(G)$ is an \emph{induced matching} if $M$ is a matching and for any two distinct edges $uv, u'v'\in M$ there are no edges of $G$ among
$\{u,u',v,v'\}$ except $uv$ and $u'v'$. In particular, in a bipartite graph this means
$uv'\notin E(G)$ and $u'v\notin E(G)$.
\end{definition}

\begin{definition}
Given graph $G=(V,E)$, the \emph{induced matching number} is the maximum size of an induced matching in $G$.
\end{definition}

\section{General Bounds}
\label{sec:bounds}

\subsection{A lower bound via induced matchings}

\begin{lemma}
\label{lem:lower}
For every bipartite graph $G$,
\[
\fp(G) \ge \nu_{\ind}(G).
\]
\end{lemma}

\begin{proof}
A Ferrers graph contains no induced $2K_2$ (Prop.~\ref{prop:one-component}), hence contains at most one edge from any induced matching.
In an edge partition into $k$ Ferrers graphs, each part contains at most one induced-matching edge, so $\nu_{\ind}(G)\le k$.
\end{proof}

\subsection{An upper bound via Dilworth chain decompositions}

For $u\in U$, let $N(u)\subseteq V$ be its neighborhood.
Let $\mathcal{N}_U=\{N(u):u\in U\}$ partially ordered by inclusion, and write $\width(\mathcal P_U)$ for the maximum size of an antichain of the corresponding neighborhood-inclusion poset (multiplicities do not affect width) on vertices of $U$ (and $V$, resp.) defined by $u\preceq u'$ iff $N(u)\subseteq N(u')$.

\begin{theorem}
\label{thm:upper}
For every bipartite graph $G=(U,V,E)$,
\[
\fp(G) \le \min\big(\width(\mathcal P_U),\,\width(\mathcal P_V)\big).
\]
\end{theorem}

\begin{proof}
By Dilworth's theorem \cite{Dilworth1950}, $\mathcal{N}_U$ can be partitioned into $w=\width(\mathcal P_U)$ chains.
This yields a partition $U=U_1\dot\cup\cdots\dot\cup U_w$ such that within each $U_i$ the neighborhoods are totally ordered by inclusion.
Let $E_i$ be the edges incident to $U_i$; then $(U,V,E_i)$ is Ferrers and $E=\dot\cup_i E_i$.
Thus $\fp(G)\le w$.
The same argument applies on $V$, giving the stated minimum.
\end{proof}

\section{Exact values and bounds for some families}
\label{sec:families}

Throughout, $P_n$ denotes the $n$-vertex path and $C_n$ the $n$-vertex cycle.

\begin{lemma}
\label{lem:path-max3}
Any Ferrers subgraph of a path or a cycle contains at most $3$ edges.
\end{lemma}

\begin{proof}
Any subgraph of a path or cycle has maximum degree at most $2$, hence every connected component is a path.
If a Ferrers subgraph contains a component isomorphic to $P_5$, then its two end edges are vertex-disjoint and there is no edge between their endpoints inside the subgraph; hence they induce a $2K_2$, a contradiction. Thus each nontrivial component has at most $3$ edges. Finally, by Prop.~\ref{prop:one-component} a Ferrers graph has at most one nontrivial component, so the entire Ferrers subgraph has at most $3$ edges.
\end{proof}

\subsection{Paths}

\begin{theorem}
\label{thm:path}
For the path $P_n$ on $n\ge 2$ vertices,
\[
\fp(P_n)=\left\lceil \frac{n-1}{3}\right\rceil
\qquad\text{and}\qquad
\nu_{\ind}(P_n)=\left\lceil \frac{n-1}{3}\right\rceil.
\]
In particular, $\fp(P_n)=\nu_{\ind}(P_n)$ for all $n$.
\end{theorem}

\begin{proof}
Let $m=n-1$ be the number of edges of $P_n$.
Partition the edges into consecutive blocks of $3$ edges (and a final block of size $1$ or $2$ if needed).
Each block induces $P_2$, $P_3$, or $P_4$, hence is Ferrers, so $\fp(P_n)\le \lceil m/3\rceil$.
By Lemma~\ref{lem:path-max3}, every Ferrers subgraph of $P_n$ has at most $3$ edges, so $\fp(P_n)\ge \lceil m/3\rceil$.

For the induced matching number, selecting every third edge along the path gives an induced matching of size $\lceil m/3\rceil$, and no induced matching can use more than one edge among any $3$ consecutive edges, so $\nu_{\ind}(P_n)=\lceil m/3\rceil$.
\end{proof}

\subsection{Even cycles}

\begin{theorem}
\label{thm:cycle}
For the cycle $C_n$ on $n\ge 4$ vertices,
\[
\fp(C_n)=\left\lceil \frac{n}{3}\right\rceil
\qquad\text{and}\qquad
\nu_{\ind}(C_n)=\left\lfloor \frac{n}{3}\right\rfloor.
\]
In particular, for even cycles $C_{2t}$ (which are bipartite),
\[
\fp(C_{2t})\in\Big\{\nu_{\ind}(C_{2t}),\,\nu_{\ind}(C_{2t})+1\Big\}.
\]
\end{theorem}

\begin{proof}
The cycle $C_n$ has $n$ edges.
Partition the edges into consecutive blocks of $3$ around the cycle (with a final block of size $1$ or $2$ if needed).
Each block induces a path on at most $4$ vertices, hence is Ferrers, so $\fp(C_n)\le \lceil n/3\rceil$.
Lemma~\ref{lem:path-max3} gives the reverse inequality.

For $\nu_{\ind}(C_n)$, selecting every third edge yields an induced matching of size $\lfloor n/3\rfloor$.
Conversely, an induced matching cannot contain two edges within the same triple of consecutive edges, so it has size at most $\lfloor n/3\rfloor$.
\end{proof}

\subsection{Ladders}

Let $L_n=P_n\square P_2$ denote the ladder graph on $2n$ vertices.
We label the vertices in the two rails by $t_1,\dots,t_n$ (top) and $b_1,\dots,b_n$ (bottom),
and the edges are the rails $t_it_{i+1}$, $b_ib_{i+1}$ for $1\le i\le n-1$, together with the rungs $t_ib_i$ for
$1\le i\le n$.

\begin{theorem}\label{thm:nuind-ladder}
For $n\ge 1$,
\[
\nu_{\ind}(L_n)=\left\lceil\frac{n}{2}\right\rceil.
\]
\end{theorem}

\begin{proof}
\emph{Lower bound.}
The set of rungs $\{t_ib_i : i\ \text{odd}\}$ is a matching of size $\lceil n/2\rceil$.
If $i$ and $j$ are distinct odd indices, then $|i-j|\ge 2$, hence there is no edge of $L_n$ between the vertex pairs
$\{t_i,b_i\}$ and $\{t_j,b_j\}$ (all edges of $L_n$ lie within a column or between adjacent columns).
Thus these rungs form an induced matching, and $\nu_{\ind}(L_n)\ge \lceil n/2\rceil$.

\emph{Upper bound.}
For an edge $e$ of $L_n$, define its \emph{start column} $s(e)$ by
\[
s(t_ib_i)=i,\qquad s(t_it_{i+1})=i,\qquad s(b_ib_{i+1})=i.
\]
We claim that if $M$ is an induced matching in $L_n$, then the start columns $\{s(e):e\in M\}$ differ by at least $2$:
there do not exist distinct $e,f\in M$ with $s(f)=s(e)+1$.

Indeed, fix $i$ and suppose $s(e)=i$ and $s(f)=i+1$.
Every edge with start column $i$ contains at least one of $t_i,b_i,t_{i+1},b_{i+1}$, and every edge with start column $i+1$
contains at least one of $t_{i+1},b_{i+1},t_{i+2},b_{i+2}$.
In all possible cases, there is an edge of $L_n$ between an endpoint of $e$ and an endpoint of $f$:
either the two edges share a vertex, or one uses $t_{i+1}$ and the other uses $b_{i+1}$ (joined by the rung $t_{i+1}b_{i+1}$),
or one uses $t_{i+1}$ and the other uses $t_{i+2}$ (joined by $t_{i+1}t_{i+2}$), or one uses $b_{i+1}$ and the other uses $b_{i+2}$
(joined by $b_{i+1}b_{i+2}$).
Any such adjacency between endpoints is forbidden in an induced matching, so $s(f)\neq s(e)+1$.

Consequently, the set $\{s(e):e\in M\}\subseteq\{1,\dots,n\}$ has pairwise gaps at least $2$, hence has size at most $\lceil n/2\rceil$.
Therefore $\nu_{\ind}(L_n)\le \lceil n/2\rceil$.
\end{proof}

\begin{corollary}\label{cor:fp-ladder-lb}
For all $n\ge 1$,
\[
\left\lceil\frac{n}{2}\right\rceil \le \fp(L_n)\le n.
\]
\end{corollary}

\begin{proof}
Lower bound immediately follows form Theorem~\ref{thm:nuind-ladder} and Lemma~\ref{lem:lower}. It is easy to see that the
Dilworth-width of the corresponding poset is $n$ and apply Theorem~\ref{thm:upper}.
\end{proof}

\begin{remark}
For the ladder graph, is easy to see that the lower bound is the truth. Consider the ladder of 8 vertices. Two letter ``P" graph put together in a right way covers the graph, and they are Ferrers. We leave the rest to the Reader.
\end{remark}

\section{Additivity and Separations}
\label{sec:separations}

\subsection{Additivity over disjoint unions}

\begin{lemma}
\label{lem:additivity}
Let $G$ and $H$ be vertex-disjoint graphs, each with at least one edge, and let $G\dot\cup H$ denote their disjoint union.
Then
\[
\fp(G\dot\cup H)=\fp(G)+\fp(H),
\qquad
\nu_{\ind}(G\dot\cup H)=\nu_{\ind}(G)+\nu_{\ind}(H).
\]
\end{lemma}

\begin{proof}
Induced matchings are additive over disjoint unions.

For $\fp$, note that if a Ferrers graph contains edges from two different components of $G\dot\cup H$, then it contains two vertex-disjoint edges with no edges between their endpoints (since the host has no such edges), hence an induced $2K_2$, impossible.
Thus every Ferrers part in a partition of $G\dot\cup H$ lies entirely within one component, and the minimum number of parts is the sum.
\end{proof}

\subsection{A growing gap between $\fp$ and $\nu_{\ind}$}

\begin{corollary}
\label{cor:gap}
Let $F_t$ be the disjoint union of $t$ copies of $C_8$.
Then
\[
\fp(F_t)-\nu_{\ind}(F_t)=t.
\]
\end{corollary}

\begin{proof}
By Theorem~\ref{thm:cycle}, $\fp(C_8)=\lceil 8/3\rceil=3$ and $\nu_{\ind}(C_8)=\lfloor 8/3\rfloor=2$.
Apply Lemma~\ref{lem:additivity}.
\end{proof}

\subsection{A separation for the Dilworth-width upper bound}

\begin{definition}
The \emph{crown graph} $H_n$ ($n\ge 3$) is obtained from $K_{n,n}$ by deleting a perfect matching:
$U=\{u_1,\dots,u_n\}$, $V=\{v_1,\dots,v_n\}$, and $u_i v_j\in E(H_n)$ iff $i\neq j$.
\end{definition}

\begin{theorem}
\label{thm:crown}
For the crown graph $H_n$ with $n\ge 3$,
\[
\fp(H_n)=2,
\qquad
\width(\mathcal P_U)=\width(\mathcal P_V)=n.
\]
In particular, the Dilworth-width bound in Theorem~\ref{thm:upper} can be off by a factor $\Theta(n)$.
\end{theorem}

\begin{proof}
Order $U$ and $V$ by indices.
Partition $E(H_n)$ into
\[
E^{+}:=\{u_i v_j : i<j\},\qquad
E^{-}:=\{u_i v_j : i>j\}.
\]
Then $E(H_n)=E^+\dot\cup E^-$.

In $G^+=(U,V,E^+)$, we have
$N_{G^+}(u_1)\supseteq N_{G^+}(u_2)\supseteq\cdots\supseteq N_{G^+}(u_n)=\varnothing$,
so $G^+$ is Ferrers; similarly $G^-$ is Ferrers.
Thus $\fp(H_n)\le 2$.
Since $H_n$ contains an induced $2K_2$ (e.g.\ edges $u_1v_2$ and $u_2v_1$), it is not Ferrers, hence $\fp(H_n)\ge 2$.

Finally, $N(u_i)=V\setminus\{v_i\}$, $i=1,\dots, n$, and these $n$ sets are pairwise incomparable, hence form an antichain of size $n$.
So $\width(\mathcal P_U)=n$, and it holds symmetrically on $V$.
\end{proof}

\begin{theorem}
\label{thm:Kmn-minus-matching}
Let $G=K_{m,n}\setminus M$, where $M$ is a matching of size $t$ in $K_{m,n}$.
Then
\[
\fp(G)=
\begin{cases}
1, & t\le 1,\\
2, & t\ge 2.
\end{cases}
\]
\end{theorem}

\begin{proof}
If $t=0$, then $G=K_{m,n}$ is Ferrers, so $\fp(G)=1$.
If $t=1$, say the missing edge is $u_1v_1$, then all neighborhoods on the $U$-side are comparable:
$N(u_1)=V\setminus\{v_1\}\subseteq V=N(u_i)$ for $i\ge 2$, hence $G$ is Ferrers and $\fp(G)=1$.

Assume now that $t\ge 2$.
After relabeling, we may assume the deleted matching is
$M=\{u_iv_i: i=1,\dots,t\}$ with $U=\{u_1,\dots,u_m\}$ and $V=\{v_1,\dots,v_n\}$.
Then $N(u_i)=V\setminus\{v_i\}$ for $i=1,\dots,t$, and these neighborhoods are pairwise incomparable for $t\ge 2$.
In particular, $G$ is not Ferrers, so $\fp(G)\ge 2$.

We show $\fp(G)\le 2$ by constructing a partition of $E(G)$ into two Ferrers subgraphs.
Let $V_0=\{v_{t+1},\dots,v_n\}$ and $U_0=\{u_{t+1},\dots,u_m\}$ denote the vertices not incident to the deleted matching.
Define
\[
E^+ := \{u_i v_j : 1\le i<j\le t\}\ \cup\ \{u_i v : 1\le i\le t,\ v\in V_0\},
\]
and let $E^-:=E(G)\setminus E^+$.
Clearly $E(G)=E^+\dot\cup E^-$.

In $G^+=(U,V,E^+)$, the $U$-side neighborhoods satisfy
\[
N_{G^+}(u_1)\supseteq N_{G^+}(u_2)\supseteq\cdots\supseteq N_{G^+}(u_t)\supseteq
N_{G^+}(u_{t+1})=\cdots=N_{G^+}(u_m)=\varnothing,
\]
so $G^+$ is Ferrers.

For $G^-=(U,V,E^-)$, consider neighborhoods on the $V$-side.
For $1\le j\le t$ we have
\[
N_{G^-}(v_j)=\{u_i: 1\le i\le t,\ i>j\}\ \cup\ U_0,
\]
while for $v\in V_0$ we have $N_{G^-}(v)=U_0$.
Hence
\[
N_{G^-}(v_1)\supseteq N_{G^-}(v_2)\supseteq\cdots\supseteq N_{G^-}(v_t)=U_0
= N_{G^-}(v)\ \ (v\in V_0),
\]
so $G^-$ is Ferrers.
Thus $\fp(G)\le 2$, and together with $\fp(G)\ge 2$ we obtain $\fp(G)=2$ for $t\ge 2$.
\end{proof}

\section{A host-induced $2K_2$ conflict graph lower bound}\label{sec:conflict}

The obstruction $2K_2$ can be viewed from two perspectives: it may be induced \emph{inside a part}, or it may be forced by missing cross-edges already in the host graph.
The latter yields a simple and useful necessary condition.

\begin{definition}
Let $G=(U,V,E)$ be bipartite. The \emph{host-induced $2K_2$ conflict graph} $\mathcal C_{\mathrm{host}}(G)$ is the graph with vertex set $E$.
Two distinct edges $e=uv$ and $f=u'v'$ (with $u,u'\in U$ and $v,v'\in V$) are adjacent if
$u\neq u'$, $v\neq v'$, and the cross-edges are \emph{absent in the host}:
$uv'\notin E$ and $u'v\notin E$.
\end{definition}

\begin{proposition}\label{prop:conflict-lb}
For every bipartite graph $G$,
\[
\chi\big(\mathcal C_{\mathrm{host}}(G)\big)\le \fp(G).
\]
\end{proposition}

\begin{proof}
Suppose $E=E_1\dot\cup\cdots\dot\cup E_k$ and each $G_i=(U,V,E_i)$ is Ferrers.
If two edges $e,f\in E$ are adjacent in $\mathcal C_{\mathrm{host}}(G)$, then the cross-edges $uv'$ and $u'v$ do not exist in $G$, hence they do not exist in any subgraph $G_i$.
Therefore $e$ and $f$ cannot lie in the same Ferrers part $E_i$, since they would form an induced $2K_2$ inside $G_i$.
Thus each $E_i$ is an independent set of $\mathcal C_{\mathrm{host}}(G)$, and so $k\ge \chi(\mathcal C_{\mathrm{host}}(G))$.
\end{proof}

\begin{remark}\label{rem:conflict-not-characterization}
In general, $\mathcal C_{\mathrm{host}}(G)$ only captures \emph{forced} $2K_2$ obstructions.
It does \emph{not} characterize feasibility of a Ferrers partition: for $G=K_{2,2}$, the set of two opposite edges is independent in $\mathcal C_{\mathrm{host}}(G)$,
yet it induces a $2K_2$ as a subgraph and hence is not Ferrers.
\end{remark}

\subsection{A worked example: the crown graph $H_4$}

Let $H_4$ be the crown graph on $U=\{u_1,u_2,u_3,u_4\}$ and $V=\{v_1,v_2,v_3,v_4\}$, i.e.\ $H_4=K_{4,4}\setminus\{u_iv_i:i\in[4]\}$.
Its bipartite adjacency matrix (rows indexed by $U$, columns by $V$) is $J-I$:
\[
A=
\begin{pmatrix}
0&1&1&1\\
1&0&1&1\\
1&1&0&1\\
1&1&1&0
\end{pmatrix}.
\]
Define $ E^+ := \{u_iv_j : 1\le i<j\le 4\}$ and
$E^- := \{u_iv_j : 1\le j<i\le 4\}$.
Then $E(H_4)=E^+\dot\cup E^-$, and both $G^+=(U,V,E^+)$ and $G^-=(U,V,E^-)$ are Ferrers:
in $G^+$ the $U$-side neighborhoods satisfy
$N(u_1)\supseteq N(u_2)\supseteq N(u_3)\supseteq N(u_4)$,
while in $G^-$ the $V$-side neighborhoods satisfy
$N(v_1)\supseteq N(v_2)\supseteq N(v_3)\supseteq N(v_4)$.
Hence $\fp(H_4)\le 2$, and since $H_4$ is not Ferrers we have $\fp(H_4)=2$.
Moreover, in this example the lower bound from Proposition~\ref{prop:conflict-lb} is tight: $\chi(\mathcal C_{\mathrm{host}}(H_4))=2$.

\section{Matrix Viewpoint}
\label{sec:matrix}

Let $A\in\{0,1\}^{m\times n}$ be the bipartite adjacency matrix of $G$.
A Ferrers graph corresponds to a matrix that can be permuted so that each row is a (possibly empty) prefix of ones.
This gives a compact matrix restatement of $\fp(G)$.

\begin{remark}
\label{rem:matrix}
$\fp(G)$ is the minimum $k$ such that the $1$-entries of $A$ can be partitioned into $k$ disjoint supports, each support forming a Ferrers $0$--$1$ matrix (allowing independent row/column permutations for each part).
\end{remark}

For the crown graph $H_n$, the matrix is $J-I$ (all ones except the diagonal).
The decomposition $E(H_n)=E^+\dot\cup E^-$ from Theorem~\ref{thm:crown} is simply the split of $J-I$ into its strictly upper- and strictly lower-triangular parts (after ordering rows/columns by index), and each part is Ferrers.
This example is a useful reminder that neighborhood-width can be a coarse proxy for $\fp$.

Theorem~\ref{thm:Kmn-minus-matching} can be phrased neatly in matrix terms.
If $A$ is the $m\times n$ all-ones matrix with zeros at the positions of a matching (i.e.\ the zero pattern is a partial permutation matrix), then the $1$-entries of $A$ can be partitioned into two Ferrers shapes unless there is at most one zero, in which case $A$ itself is Ferrers.

\section{Computational complexity}\label{sec:complexity}

\subsection{Decision problem.}

We formalize the natural decision version of $\fp$.

\begin{problem}[\textsc{Ferrers-$k$-Partition}]\label{prob:ferrers-k-partition}
Instance: A bipartite graph $G=(U,V,E)$ and an integer $k\ge 1$.
Question: Does there exist a partition $E=E_1\dot\cup\cdots\dot\cup E_k$ such that each
$G_i=(U,V,E_i)$ is Ferrers (equivalently, an induced-$2K_2$-free bipartite) graph?
\end{problem}

\begin{proposition}\label{prop:fp-in-NP}
\textsc{Ferrers-$k$-Partition} belongs to $\NP$.
\end{proposition}

\begin{proof}
A certificate is a map $c:E\to[k]$ defining $E_i=\{e\in E: c(e)=i\}$.
Given $E_i$, we can test in polynomial time whether $G_i$ is a chain graph/Ferrers graph,
for instance by attempting to produce a nested-neighborhood ordering on one side and rejecting otherwise.
Moreover, chain graphs admit linear-time \emph{certifying} recognition algorithms \cite{HeggernesKratsch2007}.
\end{proof}

Problem~\ref{prob:ferrers-k-partition} with $k=1$ is exactly recognition of chain graphs (Ferrers graphs), which is solvable in linear time with certificates \cite{HeggernesKratsch2007}.
\\

\subsection{Relation to the template-partition framework.}
According to Gy\H{o}rffy et al. for bipartite host graphs $G$ and $H=2K_2$, the $H$-free bipartite graphs are precisely Ferrers (chain) graphs,
and therefore
\[
\fp(G)=\chi'_{2K_2}(G).
\]
Their standing conventions exclude empty template graphs and isolated vertices inside templates;
this does not affect the optimum value, since empty parts can be dropped and isolated vertices removed without changing the edge partition. \cite{GyorffyLondonNagyPluhar2025}

\subsection{Hardness evidence from neighboring problems.}
We do not attempt a full complexity classification for $\fp$ here; instead we record a few related NP-hardness
results indicating that chain-structure constraints can be algorithmically rigid.
(These do \emph{not} by themselves imply NP-hardness of computing $\fp(G)$.)

\begin{itemize}[leftmargin=2em]
\item \textbf{Chain-structure completion is hard.}
The \emph{chain graph sandwich problem} (deciding whether one can choose a graph between mandatory/optional edges
that is a chain graph) is NP-complete. \cite{DantasEtAl2011}

\item \textbf{The cover analogue becomes hard for $k\ge 3$.}
In the \emph{cover} version (edges may be covered more than once),
the $k$-\emph{chain subgraph cover} problem is NP-complete for $k\ge 3$ and polynomial-time solvable for $k\le 2$;
see, e.g., Takaoka for a clear statement and references. \cite{Takaoka2016}

\item \textbf{Induced matchings are hard already on bipartite graphs.}
Computing a maximum induced matching is NP-complete for bipartite graphs. \cite{Cameron1989}
Since induced matchings provide universal lower bounds on $\fp$, this suggests that obtaining tight lower-bound
certificates may be difficult in general instances.
\end{itemize}

\section{Open Problems}

The bounds $\nu_{\ind}(G)\le \fp(G)\le \min\{\width(\mathcal P_U),\width(\mathcal P_V)\}$ are simple, but the examples already show there is room between them.
Paths satisfy $\fp=\nu_{\ind}$, even cycles can have a gap of $1$, and by disjoint unions of cycles the gap $\fp-\nu_{\ind}$ can grow arbitrarily.

\begin{problem}
Compute $\fp(G)$ for further natural bipartite families (grids, convex bipartite graphs, for instance), and identify graph-theoretic conditions forcing $\fp(G)=\nu_{\ind}(G)$.
\end{problem}

\begin{problem}
How large can $\fp(G)$ be as a function of $\nu_{\ind}(G)$ under constraints such as bounded maximum degree or given edge density?
\end{problem}

\begin{problem}
Is it true that if $\fp(G)=\nu_{\ind}(G)$, then $G$ admits an optimal Ferrers partition in which every part is a forest of diameter at most $3$ (equivalently, a bistar or $P_4$)?
\end{problem}

\begin{problem}
From algorithmic perspective, a concrete next step is to determine the complexity of $\fp(G)\le 2$ for bipartite $G$. More broadly, it would be interesting to study approximation and parameterized aspects of $\fp(G)$, and to identify additional graph classes where $\fp(G)$ is computable in polynomial time.
\end{problem}

\end{document}